\documentclass[12pt]{amsart}

\usepackage{amssymb, amsmath, xcolor}
\usepackage{import}
\usepackage{tikz}
\usetikzlibrary{shapes, patterns}
\tikzset{filled/.style={minimum width=5pt,inner sep=0pt,circle,fill=black}}
\usetikzlibrary{calc}
\usepackage{subcaption}
\usepackage{mathtools}
\usepackage[utf8]{inputenc}
\usepackage{float}
\usepackage{todonotes}
\usepackage[square,numbers]{natbib}
\usepackage{qtree}

\newtheorem{theorem}{Theorem}[section]
\newtheorem{lemma}[theorem]{Lemma}

\newtheorem{question}[theorem]{Question}

\theoremstyle{definition}
\newtheorem{definition}[theorem]{Definition}

\theoremstyle{remark}
\newtheorem{remark}[theorem]{Remark}

\numberwithin{equation}{section}
\numberwithin{figure}{section}

\renewcommand{\mod}{\operatorname{mod}}
\newcommand{\N}{\mathbb{N}}
\newcommand{\Z}{\mathbb{Z}}

\title[Chromatic Numbers]{Chromatic Numbers with Open and Nonzero Local Modular Constraints}

\author[Herden, Meddaugh, Sepanski, $\ldots$]{Daniel Herden, Jonathan Meddaugh, Mark R. Sepanski, William Clark, Adam Kraus, Ellie Matter, Kingsley Michael, Mitchell Minyard, Maricela Ramirez, Kyle Rosengartner, Elyssa Stephens, John Stephens}

\thanks{The first author was supported by Simons Foundation grant MPS-TSM-00007788.
	The second author was supported by a grant from the Simons Foundation (960812, JM)}

\address{
All authors:
Department of Mathematics,
Baylor University,
Sid Richardson Building,
1410 S.~4th Street,
Waco, TX 76706, USA}

\email{daniel\_herden@baylor.edu,  jonathan\_meddaugh@baylor.edu, mark\_sepanski@baylor.edu, william\_clark2@baylor.edu, adam\_kraus1@baylor.edu, ellie\_carr3@baylor.edu, kyle\_rosengartner1@baylor.edu, ellie\_cirillo1@baylor.edu, john\_stephens2@baylor.edu, mitch\_minyard1@baylor.edu, kingsley\_michael1@baylor.edu, maricela\_ramirez1@baylor.edu}

\date{\today}

\begin{document}


\begin{abstract}
    In this paper, we explore chromatic numbers subject to various local modular constraints. For fixed $n$, we consider proper integer colorings of a graph $G$ for which the closed and open neighborhood sums have nonzero remainders modulo $n$ and provide bounds for the associated chromatic numbers $\chi_n(G)$ and $\chi_{(n)}(G)$, respectively.  In addition, we provide bounds for $\chi_{(n,k)}(G)$, the minimal order of a proper integer coloring of $G$ with open neighborhood sums congruent to $k\mod n$ (when such a coloring exists) as well as precise values for certain families of graphs.
\end{abstract}

\maketitle

\tableofcontents

\section{Introduction}

Graph coloring is one of the most well-studied areas in graph theory. Perhaps the most well-known graph coloring problem is the problem of finding proper colorings of the vertices of a graph $G$. The minimum number of colors in such a coloring is the well-studied chromatic number of $G$, $\chi(G)$, which is, despite its long history, still an object of active research \cite{cambie2024removing, campena2023graph, cervantes2023chromatic, char2023improved, haxell2023large, heckel2021non, isaev2021chromatic, matsumoto2021chromatic}. Many other graph invariants can be defined by considering the minimum number of colors under different constraints.
As a variation on the well-studied topic of odd colorings \cite{RemarksOdd, CranstonSparse, Cranston1Planar, KnauerBoundedness, PetrChiPlanar, NeighborhoodParity}, this includes the \emph{odd-sum chromatic number} of a graph $G=(V,E)$, denoted $\chi_{\rm{os}}(G)$, which gives the minimum size of the range of a proper $\Z$-labeling $\ell:V\longrightarrow \Z$ such that at every vertex, the sum of its label along with the labels of the adjacent vertices is odd. This concept has been studied in \cite{CaroYairPetru2023, CranstonOddSum}, which includes general bounds given by Caro, Petru\v{s}evski, and \v{S}krekovski \cite{CaroYairPetru2023} as well as various bounds for certain classes of graphs. This idea has been further generalized in \cite{ClosedNeighborhoodSums} by considering proper colorings where all these neighborhood sums have remainder $k \mod n$ for some fixed integers $n$ and $k$, with associated chromatic number $\chi_{n,k}(G)$. Thus, we have $\chi_{2,1}(G)=\chi_{\rm{os}}(G)$.


In Sections 3 and 4 of this paper, we consider an alternative generalization of this notion. In particular, we note that having remainder $1 \mod 2$ is the same as \emph{not} having remainder $0\mod 2$. Of course, for $n>2$, these ideas are not equivalent and we develop the theory of proper colorings where no neighborhood sum has remainder $0 \mod n$. We provide bounds and specific values for some families of graphs for the associated chromatic number $\chi_n(G)$.

In the remainder of the paper, we consider \emph{open neighborhood sums}, i.e., the sums over vertices adjacent, but \emph{not} equal to, a fixed vertex. In Sections 5 and 6, we develop results for $\chi_{(n,k)}(G)$, the minimum size of the range of a proper $\Z$-labeling of a graph $G$ such that all open neighborhood sums have remainder $k\mod n$; and in Sections 7 and~8, we consider $\chi_{(n)}(G)$, the minimum range size of a proper $\Z$-labeling of a graph $G$ with no open neighborhood sum congruent to $0$ modulo $n$.

\section{Definitions}

We write $\N$ for the nonnegative integers and $\Z^+$ for the positive ones.
For $a,b \in \Z$, not both zero, we write $(a,b)$ for the greatest common divisor of $a$ and $b$.
For $k\in\Z$ and $n\in\Z^+$, we write $[k]$ for the image of $k$ in $\Z_n$.

We write $G=(V,E)$ for a simple graph with vertex set $V$ and edge set $E$ and \[\ell:V\longrightarrow \Z\] for a \emph{coloring} or \emph{labeling} of the vertices by $\Z$, also called a \emph{$\Z$-labeling}.
The \emph{order} of a labeling, $|\ell|$, is the size of its range.

If $v\in V$,
the \emph{open neighborhood} of $v$, $N(v)$, consists of all vertices adjacent to $v$, and
the \emph{closed neighborhood} of $v$, $N[v]=N(v) \cup \{v\}$, consists of $v$ and all vertices adjacent to $v$.
A labeling is called \emph{proper} if $\ell(v)\neq\ell(w)$ for each $v\in V$ and each $w\in N(v)$. The \emph{chromatic number} of $G$, $\chi(G)$, is the minimum order of a proper labeling of $G$.

Let $k\in\Z$ and $n\in\Z^+$.
In \cite{ClosedNeighborhoodSums}, we have investigated the following notion:
    A \emph{closed coloring with remainder $k\mod n$} of $G$ is a $\Z$-labeling $\ell$ of $G$ so that, for each $v\in V$,
    \[ \sum_{w\in N[v]} \ell(w) \equiv k \mod n. \]
    If no proper closed coloring with remainder $k\mod n$ of $G$ exists, we say that $\chi_{n,k}(G)$ does not exist.
    Otherwise, if proper closed colorings with remainder $k\mod n$ of $G$ exist of finite order, the \emph{closed chromatic number of $G$ with remainder $k \mod n$}, written \[\chi_{n,k}(G),\] is the minimum order of a proper closed coloring with remainder $k\mod n$ of $G$. If such colorings exist only of infinite order, we write ${\chi_{n,k}(G)=\infty}$.

    This notion arose as a generalization of the \emph{odd-sum chromatic number} of $G$, $\chi_{\rm{os}}(G)$, introduced in \cite{CaroYairPetru2023}. With the above notation, $\chi_{2,1}(G) = \chi_{\rm{os}}(G)$.
    However, there is another natural generalization of $\chi_{\rm{os}}(G)$ based on the observation that $[1]$ is the only nonzero element in $\Z_2$.

\begin{definition}\label{def: closed coloring with nonzero remainder}
    Let $k\in\Z$ and $n\in\Z^+$.
    A \emph{closed coloring with nonzero remainders $\mod n$} of $G$ is a $\Z$-labeling $\ell$ of $G$ so that, for each $v\in V$,
    \[ \sum_{w\in N[v]} \ell(w) \not\equiv 0 \mod n. \]
    If no proper closed coloring with nonzero remainders $\mod n$ of $G$ exists, we say that $\chi_{n}(G)$ does not exist.
    Otherwise, if proper closed colorings with nonzero remainders $\mod n$ of $G$ exist of finite order, the \emph{closed chromatic number of $G$ with nonzero remainders $ \mod n$}, written \[\chi_{n}(G),\] is the minimum order of a proper closed coloring with nonzero remainders $\mod n$ of $G$. If such colorings exist only of infinite order, we write $\chi_{n}(G)=\infty$.
\end{definition}

With Definition \ref{def: closed coloring with nonzero remainder} in hand, we see that
\[ \chi_{\rm{os}}(G) = \chi_{2,1}(G) = \chi_2(G).\]

Note that in all of the above definitions, closed neighborhoods were studied. However, there are analogous definitions with open neighborhoods.
\begin{definition}\label{def: open colorings}
    Let $k\in\Z$ and $n\in\Z^+$.
    \begin{itemize}
    \item An \emph{open coloring with remainder $k\mod n$} of $G$ is a $\Z$-labeling $\ell$ of $G$ so that, for each $v\in V$,
    \[ \sum_{w\in N(v)} \ell(w) \equiv k \mod n. \]
    If no proper open coloring with remainder $k\mod n$ of $G$ exists, we say that $\chi_{(n,k)}(G)$ does not exist.
    Otherwise, if proper open colorings with remainder $k\mod n$ of $G$ exist of finite order, the \emph{open chromatic number of $G$ with remainder $k \mod n$}, written \[\chi_{(n,k)}(G),\] is the minimum order of a proper open coloring with remainder $k\mod n$ of $G$. If such colorings exist only of infinite order, we write $\chi_{(n,k)}(G)=\infty$.

    \item An \emph{open coloring with nonzero remainders $\mod n$} of $G$ is a $\Z$-labeling $\ell$ of $G$ so that, for each $v\in V$,
    \[ \sum_{w\in N(v)} \ell(w) \not\equiv 0 \mod n. \]
    If no proper open coloring with nonzero remainders $\mod n$ of $G$ exists, we say that $\chi_{(n)}(G)$ does not exist.
    Otherwise, if proper open colorings with nonzero remainders $\mod n$ of $G$ exist of finite order, the \emph{open chromatic number of $G$ with nonzero remainders $ \mod n$}, written \[\chi_{(n)}(G),\] is the minimum order of a proper open coloring with nonzero remainders $\mod n$ of $G$. If such colorings exist only of infinite order, we write $\chi_{(n)}(G)=\infty$.
    \end{itemize}
\end{definition}

\section{Basic Results for $\chi_{n}(G)$} \label{sec:3}

\begin{theorem}\label{thm: nonzero mod n inequalities}
    Let $m,n\in\Z^+$ with $m\mid n$. If $\chi_m(G)$ exists, then
    \[ \chi(G) \leq \chi_n(G) \leq \chi_m(G). \]
\end{theorem}

\begin{proof}
    The first inequality follows immediately from the definition. For the second, observe that if $\ell$ is a closed coloring with nonzero remainders $\mod m$, then it is also one for $\mod n$.
\end{proof}

\begin{question}
For finite graphs $G$, it is known, \cite[Proposition 3.1]{CaroYairPetru2023}, that $\chi_{\rm{os}}(G)=\chi_2(G)$ always exists.
Therefore, Theorem \ref{thm: nonzero mod n inequalities} shows that $\chi_n(G)$ always exists for all finite graphs $G$ and even $n$. Is this also true for all odd $n$?
\end{question}

\begin{theorem}\label{thm: nonzero remainder mod n existence criteria and upper bound by n chi}
    Let $n\in\Z^+$, and let $\chi(G)$ be finite. Then a proper closed coloring with nonzero remainders $\mod n$ of $G$ exists if and only if a closed coloring with nonzero remainders $\mod n$ of $G$ exists. In that case,
    \[ \chi(G)\le \chi_{n}(G) \leq n \, \chi(G). \]
    More precisely, if $\ell$ is a closed labeling with nonzero remainders $\mod n$, then
    \[ \chi(G) \le \chi_{n}(G) \leq |\ell| \, \chi(G).\]
\end{theorem}

\begin{proof}
    Let $\ell$ be a closed coloring with nonzero remainders $\mod n$ of $G$ and let $\ell'$ be a minimal proper labeling of $G$. We may assume that the range of $\ell$ sits in $[0,n-1]$, and we may assume that the range of $\ell'$ sits in $n\Z$. Then the labeling $\ell + \ell'$ is a proper closed coloring with nonzero remainders $\mod n$ of $G$. As its order is bounded by $|\ell| \chi(G)$ and since $|\ell|\leq n$, we are done.
\end{proof}

\begin{theorem}\label{thm: regular chi n}
    Let $n,j\in\Z^+$, and let $G$ be a $j$-regular graph. Then
    \[ n \nmid (j+1) \implies \chi_{n}(G) = \chi(G). \]
\end{theorem}

\begin{proof}
    If $n \nmid (j+1)$, then a constant labeling of $G$ by $1$ is a closed coloring with nonzero remainders $\mod n$. Furthermore, note that $\chi(G)\le j+1$ for any $j$-regular graph~$G$. Theorem \ref{thm: nonzero remainder mod n existence criteria and upper bound by n chi} finishes the proof.
\end{proof}

For our next discussion, we recall the definition of an efficient dominating set from \cite[Section 3]{bakkerIEDS}.

\begin{definition}\label{def: independent and dominating sets}
    Let $U\subseteq V$ for a graph $G = (V,E)$. We say that $U$ is
    \begin{itemize}
        \item an \emph{efficient dominating set} if $|N(v)\cap U|=1$ for every $v\in V\backslash U$.
        \item an \emph{independent efficient dominating set} (IEDS) if $|N[v]\cap U|=1$ for every $v\in V$, i.e., it is an independent set and an efficient dominating set.
    \end{itemize}
    We say that a graph $G$ \emph{admits an IEDS} if such a collection of vertices exists for $G$.
\end{definition}

It has been shown by Bakker and van Leeuwen \cite[Theorem 3.3]{bakkerIEDS} that determining whether an arbitrary graph $G$ admits an IEDS is NP-complete. In the same paper, they also provide a linear-time algorithm that determines whether any given finite tree admits an IEDS.

\begin{lemma}\label{lem: Colorings given IEDS}
    If $G=(V,E)$ admits an IEDS $U\subseteq V$ and $\chi(G)<\infty$, then $\chi_{n}(G)$ exists for all $n\in\Z^+$. In particular,
    \[ \chi(G) \leq \chi_{n}(G) \leq \chi(G)+1.\]
    If $U$ can be colored with a single color in some minimal proper labeling of $G$ such that $U$ contains all vertices of that color, then the inequality improves to \[\chi_{n}(G) = \chi(G).\]

\end{lemma}

\begin{proof}
    Let $U$ be an IEDS for $G$.
    Write $\ell$ for a minimal proper labeling of $G$ and suppose its range lies in $n\Z \cap (1, \infty)$. The proof is finished by defining a closed coloring $\ell'$ with nonzero remainders $\mod n$ of $G$ via
    \[
        \ell'(v) =
        \begin{cases}
            \ell(v), & \text{if } v\in V\backslash U,\\
            1, & \text{if } v\in U.\qquad \qedhere
        \end{cases}
    \]
\end{proof}

\section{Examples for $\chi_{n}(G)$} \label{sec:4}

We begin with the \emph{path on $m$ vertices}, $P_m$, the \emph{complete graph on $m$ vertices}, $K_m$, the \emph{cycle on $m$ vertices}, $C_m$, and the \emph{star on $m+1$ vertices}, $S_m$.

\begin{theorem}\label{thm: k mod n calc for P_m}
    Let $n,m\in \Z^+$ with $n,m\geq 2$. Then
    \begin{equation*}
         \chi_n(P_m)=
        \begin{cases}
        2, & \text{if } n\geq 3 \text{ or } n = 2 \text{ with } m\leq 3, \\
        3, & \text{if } n = 2 \text{ and } m \geq 4.
        \end{cases}
    \end{equation*}
\end{theorem}
\begin{proof}
    For $n\geq 3$, we may label the vertices alternating between $0$ and $1$ for a proper coloring. Then the closed neighborhood
    sums are $1$ and $2$, and we have a proper closed coloring with nonzero remainders $\mod n$.

    If $n = 2$, proper closed $2$-colorings with nonzero remainders $\mod 2$ for $m=2$ and $m=3$ are provided by $(0,1)$ and $(0,1,0)$, respectively.
    For $m\geq 4$, we first show that $\chi_{n}(P_m) > 2$.
    If not, there is a proper closed $2$-coloring with nonzero remainders $\mod 2$ of the form $(a,b,a,b,\ldots)$. The neighborhood sums for the second and third vertex give $b\equiv 2a+b \equiv 1 \mod 2$ and $a\equiv a+2b \equiv 1 \mod 2$, respectively. Thus, we find $a+b \equiv 0\not\equiv 1 \mod 2$ for the neighborhood sum for the first vertex, a contradiction.
    It remains to exhibit a proper closed $3$-coloring with nonzero remainders $\mod 2$ of $P_m$.
    If $m \equiv 0 \mod 3$, then one such coloring is provided by $(0,1,2,0,1,2,\ldots, 0,1,2)$.
    If $m \not\equiv 0 \mod 3$, then $(1,2,0,1,2,\ldots)$ works.
\end{proof}

\begin{theorem}\label{thm: k mod n calc for K_m}
    Let $n,m\in\Z^+$ with $n\geq 2$. Then
    \begin{equation*}
        \chi_n(K_m)=m.
    \end{equation*}
\end{theorem}
\begin{proof}
     A labeling of the vertices with $1$ and $n, 2n, \dots, (m-1)n$ gives the result.
\end{proof}

\begin{theorem}\label{thm: chi n calc for C_m}
    Let $n,m\in\Z^+$ with $m\geq 3$ and $n\geq 2$. Then    
    \[
        \chi_n(C_m) = \chi(C_m).
    \]
\end{theorem}

\begin{proof}
    For $n\not=3$, the result follows from Theorem \ref{thm: regular chi n}.
    For $n=3$, the bound $\chi_n(C_m) \ge \chi(C_m)$ arises from Theorem \ref{thm: nonzero mod n inequalities}. Equality may be achieved by a labeling that alternates between $0$ and $1$ when $m$ is even. For $m$ odd, a labeling that starts with a $3$ and then alternates between $0$ and $1$ will work.
\end{proof}

\begin{theorem}
    Let $n,m\in\Z^+$ with $n \geq 2$. Then
    \[ \chi_n(S_m) = 2.\]
\end{theorem}

\begin{proof}
    A labeling of the central vertex with $1$ and the circumferential vertices with $0$ works.
\end{proof}

Recall that the \emph{friendship graph}, $F_m$, consists of $m$ copies of $C_3$ joined at a single vertex.

\begin{theorem}
    Let $n,m\in\Z^+$ with $n\geq2$. Then
    \[ \chi_{n}(F_m) = 3.\]
\end{theorem}
\begin{proof}
    Label the central vertex with $1$
    and label the remaining two vertices of each $C_3$ with $0$ and $n$. As $\chi(F_m) = 3$, Theorem \ref{thm: nonzero mod n inequalities} finishes the proof.
\end{proof}

Next, we turn to the \emph{complete bipartite graph}, $K_{i,j}$, with parts of sizes $i$ and $j$.

\begin{theorem} Let $n,i,j \in \Z^+$ with $n \geq 2$. Then \[\chi_n(K_{i,j}) = 2.\]
\end{theorem}

\begin{proof}
    Let $V_1$ and $V_2$ with $|V_1|=i$ and $|V_2|=j$ denote the vertex sets belonging to the two parts of $K_{i,j}$.
    Labeling all vertices of $V_1$ and $V_2$ with $1$ and $0$, respectively, works unless $n\mid i$. Similarly,
    labeling the vertices of $V_1$ and $V_2$ with $0$ and $1$, respectively, works unless $n\mid j$. Finally, if $n\mid i$ and $n\mid j$, then labeling the vertices of $V_1$ and $V_2$ with $1$ and $n+1$, respectively, works.
\end{proof}

Next, turn to the \emph{complete $m$-ary tree of height $d$}, written~$T_{m,d}$.

\begin{theorem}
 Let $n,m,d\in\Z^+$ with $n,m,d\geq 2$. Then
 \begin{equation*}
 \chi_n(T_{m,d})=\begin{cases}
     2, & \text{if } n\geq 4 \text{ or } n \nmid (m+1) \text{ or } d=2,\\
     3, & else.
 \end{cases}
 \end{equation*}
\end{theorem}

\begin{proof}
    If $n \nmid (m+1)$, a constant row labeling alternating between labels $1$ and $0$ starting from the root shows $\chi_n(T_{m,d})=2$.  If $n \mid (m+1)$ and $d=2$, a constant row labeling alternating between $0$ and $1$ starting from the root will do.
    If $n\mid(m+1)$ and $d\ge 3$, then $\chi_n(T_{m,d})=2$ is only possible for a constant row labeling alternating between suitably chosen labels $a$ and $b$ starting from the root. This generates closed neighborhood sums congruent to $a-b$, $a$, $b$, and $a+b \mod n$. Finding nonzero $a,b\in\Z_n$ so that $a\not= \pm b$ is possible if and only if $n \geq 4$. 

    We are reduced to the case of $n\mid(m+1)$, $d\ge 3$, and $n=2,3$, which will require at least three colors. In this case, a constant row labeling cycling between labels $0,1,n$ from the root works for $d\not\equiv 0 \mod 3$, while a constant row labeling cycling between labels $1,n,0$ from the root works for $d\equiv 0\mod 3$.
\end{proof}

Next we look at the \emph{regular, infinite tilings of the plane}. Write $R_3$, $R_4$, and $R_6$ for the tilings by regular triangles, squares, and hexagons, respectively.
\begin{theorem}\label{thm: regular infinite tilins of plane x n k}
    Let $n\in\Z^+$ with $n\geq 2$. Then
    \begin{align*}
    \chi_{n} (R_3) = 3,\\
    \chi_{n} (R_4) = 2,\\
    \chi_{n} (R_6) = 2.
    \end{align*}
\end{theorem}

\begin{proof}
    In each case, we have $\chi_n(G) \ge \chi(G)$ by Theorem \ref{thm: nonzero mod n inequalities} and will give a closed coloring with nonzero remainders $\mod n$ to show equality.

    For $R_3$, a proper $3$-coloring with labels $\alpha,\beta, \gamma\in \Z$ results in neighborhood sums of
    \begin{align*}
       \alpha + 3\beta + 3\gamma, \\
       3\alpha + \beta + 3\gamma, \\
       3\alpha + 3\beta + \gamma.
    \end{align*}
    To see $\chi_n(R_3)=3$, use
    $(\alpha,\beta,\gamma)=(1,0,n)$ for $n\ne 3$ and $(\alpha,\beta,\gamma)=(1, 4, 7)$ for $n=3$, respectively.

    For $R_4$ and $R_6$, a proper $2$-coloring with labels $\alpha,\beta\in \Z$ results in neighborhood sums of
    \begin{align*}
        \alpha + q\beta,  \\
        q\alpha +\beta,
    \end{align*}
    where $q=4$ for $R_4$ and $q=3$ for $R_6$, respectively. To see $\chi_n(R_4)=\chi_n(R_6)=2$, use
    $(\alpha,\beta)=(1,0)$ for $n \nmid q$ and $(\alpha,\beta)=(1, n+1)$ for $n\mid q$.
\end{proof}

Write $G(m,j)$ for the \emph{generalized Petersen graph} where $m,j\in\Z^+$ with $m\geq 3$ and $1\leq j < \frac{m}{2}$. We will use the notation $V=\{v_i, u_i \mid 0\le i< m\}$ for the vertex set of $G(m,j) =(V,E)$ with corresponding edge set
\[ E= \{v_i v_{i+1}, v_i u_i, u_i u_{i+j} \mid 0\le i < m\}, \]
where subscripts are to be read modulo $m$.
We may refer to the $v_i$ as the \emph{exterior vertices} and the $u_i$ as the \emph{interior vertices}.

\begin{theorem} \label{thm: chi n gen petersen}
    Let $n,m,j\in\Z^+$ with $n\geq 2$, $m\ge 3$, and $1\le j<\frac{m}{2}$. Then
    \[ \chi_n(G(m,j)) = \chi(G(m,j))\]
    if $n\not=2,4$.

    When $n=2,4$,
    \[ \chi_n(G(m,j)) = \chi(G(m,j))=2\]
    if $2\mid m$ and $2\nmid j$. Otherwise,
    \[ 3 \leq \chi_n(G(m,j)) \leq 6.\]
\end{theorem}

\begin{proof}
    If $n\not=2,4$, the result follows from Theorem \ref{thm: regular chi n}. 
    If $n=2,4$ with $2\mid m$ and $2\nmid j$, the $2$-coloring $\ell: V \to \{0,1\}$ that satisfies $\ell(v_i) \equiv i\mod 2$ and $\ell(u_i) \equiv (i+1)\mod 2$ shows $\chi_n(G(m,j))=2$. In all remaining cases, $\chi(G(m,j))=3$, and a labeling of the exterior vertices with $1$ and the interior vertices with $0$ provides a closed coloring with nonzero remainders $\mod n$. Theorem \ref{thm: nonzero remainder mod n existence criteria and upper bound by n chi} finishes the proof.
\end{proof}

\begin{remark}
Regarding the case of $n=2,4$ and either $2\nmid m$ or $2|j$ in Theorem~\ref{thm: chi n gen petersen}, it can be seen that there are examples for which $\chi_n(G(m,j)) \not = \chi(G(m,j))$. As a case in point, with $n=2$, $m=6$, and $j=2$, it can be seen that there are exactly four distinct closed labelings $\ell: V\to \{0,1\}$ with remainder $1\mod 2$: $0$s on the exterior vertices with $1$s on the interior vertices, $1$s on the exterior vertices with $0$s on the interior vertices, or $0$s and $1$s alternating on the exterior and (offset) interior vertices so that one inner triangle is labeled with $1$s and the other with $0$s for two more labelings. From this, it quickly follows that $\chi_2(G(6,2))=5$. However, it is known that $\chi(G(6,2))=3$.
\end{remark}

\section{Basic Results for $\chi_{(n,k)}(G)$}

We will see in Theorem \ref{thm: paths} that $\chi_{(n,k)}(G)$ may not exist. If it exists, though, we certainly have
\[ \chi(G) \leq \chi_{(n,k)}(G).\]
However,
as seen from the following theorem, the case of $k=0$ does not provide a new invariant.
\begin{theorem}\label{thm:5.1}
    Let $n\in\Z^+$. If $\chi(G)$ is finite, then \[\chi_{(n,0)}(G) = \chi(G).\]
\end{theorem}

\begin{proof}
    It suffices to provide a coloring that shows $\chi_{(n,0)}(G) \leq \chi(G)$. For this, choose a minimal order proper labeling $\ell: V\to \Z$ of $G$. Define a new labeling $\ell'$ of $G$ by $\ell'(v) = n\ell(v)$ for each $v\in V$. As this is a proper open coloring with remainder $0\mod n$ of $G$, we are done.
\end{proof}

Accordingly, for $\chi_{(n,k)}(G)$, we will often only consider the case of $k\not\equiv 0 \mod n$ for the rest of this paper.

By canceling common summands, we immediately get the following result on symmetric differences.
\begin{lemma}\label{lem: venn}
    If $\ell$ is an open coloring with remainder $k\mod n$ of $G = (V,E)$ and $v,w\in V$, then
    \[ \sum_{u\in N(v)\backslash N(w)} \ell(u)
        \equiv \sum_{u\in N(w)\backslash N(v)} \ell(u)\, \mod n. \]
\end{lemma}

Next is a result on elementary operations. 
\begin{theorem}\label{thm: k mod n units and divisor first relations}
    Let $k,u,v,c,k_1,k_2\in\Z$ and $d,m,n\in\Z^+$. If the right-hand side of each displayed equation below exists, we have the following:
    \begin{itemize}
    \item If $[u]$ is a unit in $\Z_n^\times$, then
        \[ \chi_{(n,uk)}(G) = \chi_{(n,k)}(G). \]
    \item More generally,
        \[ \chi_{(n,vk)}(G) \leq \chi_{(n,k)}(G). \]
    \item If $d$ is a common divisor of $k$ and $n$, then
        \[ \chi_{(n,k)}(G) \leq \chi_{(\frac{n}{d}, \frac{k}{d})}(G).\]
    \item If $m$ divides $n$, then
        \[ \chi_{ (m, k)}(G) \leq \chi_{(n,k)}(G).\]
    \item If $G$ admits a constant open labeling with remainder $c \mod n$, then
        \[ \chi_{(n,k - c)}(G) = \chi_{(n,k)}(G).\]
    \item Finally,
        \[ \chi_{(n,k_1+k_2)}(G) \leq \chi_{(n,k_1)}(G) \chi_{(n,k_2)}(G).\]
    \end{itemize}
\end{theorem}

\begin{proof}
    For the fourth statement,
    let $\ell$ be a minimal order proper open coloring with remainder $k\mod n$ of $G$. As this is also a proper open coloring with remainder $k\mod m$ of $G$, we are done.
    For the third statement, let $\ell$ be a minimal order proper open coloring with remainder $\frac{k}{d}\mod \frac{n}{d}$ of $G$. Define a new coloring $\ell'$ of $G$ by $\ell'(v) = d\ell(v)$ for each $v\in V$. As this is a proper open coloring with remainder $k\mod n$ of $G$, we are done.
    The first statement follows by multiplying appropriate open colorings of $G$ by $u$ or its inverse $\mod n$, and the second statement follows similarly, using Theorem \ref{thm:5.1} for $v=0$.
    For the fifth statement, note that adding and subtracting the constant open coloring leads from any minimal order proper open coloring with remainder $k\mod n$ of $G$ to proper open colorings of $G$ with remainders $(k+c)\mod n$ and $(k-c) \mod n$, respectively. For the last statement, let $\ell_1$ and $\ell_2$ be minimal order proper open colorings of $G$ with remainders $k_1\mod n$ and $k_2\mod n$, respectively. Fix any injective map $\iota:\Z\times \Z \to \Z$ such that $\iota(z_1,z_2) \equiv (z_1+z_2) \mod n$ for all $z_1,z_2\in \Z$, and define $\ell'(v) = \iota(\ell_1(v), \ell_2(v))$ for each $v\in V$ for a proper open coloring $\ell'$ with remainder $(k_1+k_2)\mod n$ of $G$.
\end{proof}

The next result can be proven in a way similar to Theorem~\ref{thm: nonzero remainder mod n existence criteria and upper bound by n chi}.

\begin{theorem}\label{thm: remainder k mod n existence criteria and upper bound by n chi}
    Let $k\in\Z$ and $n\in\Z^+$, and let $\chi(G)$ be finite. Then a proper open coloring with remainder $k\mod n$ of $G$ exists if and only if an open coloring with remainder $k\mod n$ of $G$ exists. In that case,
    \[ \chi(G)\le \chi_{(n,k)}(G) \leq n \, \chi(G). \]
    More precisely, if $\ell$ is an open coloring with remainder $k \mod n$ of $G$, then
    \[ \chi(G)\le \chi_{(n,k)}(G) \leq |\ell| \, \chi(G).\]
\end{theorem}


For our last result, we turn again to regular graphs.
\begin{theorem}\label{thm: regular chi k n}
    Let $k\in \Z$ and $n,j\in\Z^+$, and let $G=(V,E)$ be a $j$-regular graph. Then
    \[ (j, n) \mid k \implies \chi_{(n,k)}(G) = \chi(G) \]
    and, if $G$ is finite,
    \[ (j, n) \nmid k|V|
        \implies \chi_{(n,k)}(G) \text{ does not exist.} \]
\end{theorem}

\begin{proof}
    If $(j, n) \mid k$, then $jx \equiv k \mod n$ can be solved. In that case, a constant labeling of $G$ by $x$ is an open coloring with remainder $k \mod n$. Theorem \ref{thm: remainder k mod n existence criteria and upper bound by n chi} finishes the proof.

    Now suppose there is an open labeling $\ell$ of $G$ with remainder $k \mod n$, but $(j, n) \nmid k|V|$.
    Let
    \[ S = \sum_{v\in V} \sum_{u\in N(v)} \ell(u). \]
    Then $S \equiv k|V|\mod n$ as $\sum_{u\in N(v)} \ell(u) \equiv k\mod n$ for all $v\in V$.
    But each $v\in V$ is in exactly $j$ open neighborhoods. Therefore,
    $S = j\sum_{v\in V} \ell(v)$. As a result, the equation $jx \equiv k|V| \mod n$ can be solved. As this happens if and only if $(j,n)\mid k|V|$, we are done.
\end{proof}

\section{Examples for $\chi_{(n,k)}(G)$}

\begin{theorem}\label{thm: paths}
    Let $k\in\Z$ and $n,m\in \Z^+$ with $n,m\geq 2$ and $k\not\equiv 0 \mod n$. For paths,
    $\chi_{(n,k)}(P_2)=2$,
    \begin{equation*}
        \chi_{(n,k)}(P_3) =
        \begin{cases}
            2, & \text{if } (2,n) \mid k,\\
            3, & \text{otherwise},
        \end{cases}
    \end{equation*}
    and $\chi_{(n,k)}(P_4)=3$.
    For $m\geq 5$,
    \begin{equation*}
        \chi_{(n,k)}(P_m) =
        \begin{cases}
            3, & \text{if } m\equiv 3 \mod 4 \text{ and } (2,n) \mid k,\\
            \text{does not exist}, & \text{if } m\equiv 1 \mod 4, \\
            4, & \text{otherwise}.
        \end{cases}
    \end{equation*}
\end{theorem}

\begin{proof} Beginning with the first vertex, any open labeling with remainder $k\mod n$ of $P_m$ forces the labels to be congruent $\mod n$ to a repeating pattern of $(a,k,k-a,0,\ldots)$, where the variable $a\in\Z$ denotes the label of the first vertex.
    If $m\equiv 0\mod 4$, the neighborhood sum of the final vertex forces $(k-a)\equiv k \mod n$, hence $a\equiv 0 \mod n$ for a repeating pattern of $(0,k,k,0,\ldots)$. By adding $n$ where necessary, this can be made minimally proper with $3$ colors when $m=4$ and, otherwise, requires $4$ colors.
    If $m\equiv 2\mod 4$, the final vertex forces $a\equiv k \mod n$ for a repeating pattern of $(k,k,0,0,\ldots)$. By adding $n$ where necessary, this can be made minimally proper with $2$ colors when $m=2$ and, otherwise, requires $4$ colors.

    If $m\equiv 3\mod 4$, the neighborhood sum of the final vertex adds no additional constraints on $a$. In this case, $a$ may be chosen such that $a\equiv (k-a) \mod n$ if and only if $(2,n)\mid k$.
    For such a choice of $a$, the repeating pattern is $(a,k,a,0,\ldots)$. By adding $n$ where necessary, this can be made minimally proper with $2$ colors when $m=3$ and, otherwise, requires $3$ colors.
    If $(2,n)\nmid k$, then the repeating pattern is $(a,k,k-a,0,\ldots)$ for some $a\in \Z$, though $a\not\equiv (k-a) \mod n$.
    By adding $n$ where necessary, this can be made minimally proper with $3$ colors when $m=3$ and, otherwise, requires $4$ colors.

    If $m\equiv 1\mod 4$, the neighborhood sum of the final vertex forces $0\equiv k \mod n$, which violates $k \not\equiv 0\mod n$.
\end{proof}

\begin{question}
Theorem \ref{thm: paths} shows that $\chi_{(n,k)}(P_{4i+1})$ does not exist for any $n,k$, except for $k\equiv 0\mod n$. In fact, there are many graphs that share this property. By analogous arguments as above, it is straightforward to show that $K_1$, the Cartesian products $P_{4i + 1} \scalebox{.8}{$\square$} P_{4j + 1}$, and the graphs in Figure \ref{fig: nonexist examples} all share this trait. It would be interesting to find conditions on a graph $G$ that are equivalent to $\chi_{(n,k)}(G)$ existing for no $n,k$ with $k\not\equiv 0\mod n$.

    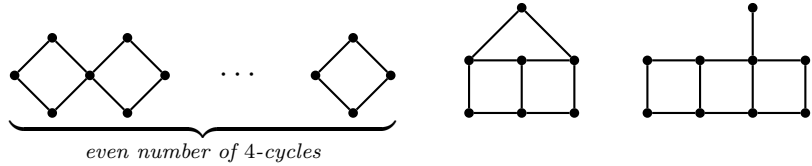
\begin{figure}[H]
    \begin{center}
\[
\underbrace{
\begin{tikzpicture}[thick,scale=.5]
          \node[circle, draw, fill, inner sep=1pt] (a) at (0,0) {};
          \node[circle, draw, fill, inner sep=1pt] (b) at (1,1) {};
          \node[circle, draw, fill, inner sep=1pt] (c) at (2,0) {};
          \node[circle, draw, fill, inner sep=1pt] (d) at (1,-1) {};
          \draw (a) -- (b) -- (c) -- (d) -- (a);
          \node[circle, draw, fill, inner sep=1pt] (B) at (3,1) {};
          \node[circle, draw, fill, inner sep=1pt] (C) at (4,0) {};
          \node[circle, draw, fill, inner sep=1pt] (D) at (3,-1) {};
          \draw (c) -- (B) -- (C) -- (D) -- (c);
          \node at (6,0) {$\cdots$};
          \node[circle, draw, fill, inner sep=1pt] (w) at (8,0) {};
          \node[circle, draw, fill, inner sep=1pt] (x) at (9,1) {};
          \node[circle, draw, fill, inner sep=1pt] (y) at (10,0) {};
          \node[circle, draw, fill, inner sep=1pt] (z) at (9,-1) {};
          \draw (w) -- (x) -- (y) -- (z) -- (w);
        \end{tikzpicture}
        }_{\text{even number of $4$-cycles}}
        \qquad
        \begin{tikzpicture}[thick,scale=.7]
          \node[circle, draw, fill, inner sep=1pt] (a) at (0,0) {};
          \node[circle, draw, fill, inner sep=1pt] (b) at (1,0) {};
          \node[circle, draw, fill, inner sep=1pt] (c) at (1,1) {};
          \node[circle, draw, fill, inner sep=1pt] (d) at (0,1) {};
          \node[circle, draw, fill, inner sep=1pt] (e) at (1,2) {};
          \node[circle, draw, fill, inner sep=1pt] (f) at (2,0) {};
          \node[circle, draw, fill, inner sep=1pt] (g) at (2,1) {};
          \draw (a) -- (b) -- (c) -- (d) -- (a);
          \draw (c) -- (g) -- (f) -- (b);
          \draw (g) -- (e) -- (d);
        \end{tikzpicture}
        \qquad
        \begin{tikzpicture}[thick,scale=.7]
          \node[circle, draw, fill, inner sep=1pt] (a) at (0,0) {};
          \node[circle, draw, fill, inner sep=1pt] (b) at (1,0) {};
          \node[circle, draw, fill, inner sep=1pt] (c) at (1,1) {};
          \node[circle, draw, fill, inner sep=1pt] (d) at (0,1) {};
          \node[circle, draw, fill, inner sep=1pt] (e) at (2,0) {};
          \node[circle, draw, fill, inner sep=1pt] (f) at (2,1) {};
          \node[circle, draw, fill, inner sep=1pt] (g) at (2,2) {};
          \node[circle, draw, fill, inner sep=1pt] (h) at (3,0) {};
          \node[circle, draw, fill, inner sep=1pt] (i) at (3,1) {};
          \draw (a) -- (b) -- (c) -- (d) -- (a);
          \draw (b) -- (e) -- (f) -- (c);
          \draw (e) -- (h) -- (i) -- (f) -- (g);
        \end{tikzpicture}
    \]
    \caption{Examples of graphs $G$, for which $\chi_{(n,k)}(G)$ exists for no $n,k$ with $k\not\equiv 0\mod n$.}
    \label{fig: nonexist examples}
    \end{center}
    \end{figure}
\end{question}

\begin{theorem} \label{Complete Graphs - Open Neighborhoods}
Let $k \in \Z$ and $n,m\in\Z^+$ with $n\geq 2$, $k\not\equiv 0\mod n$, and $m\geq 2$. For the complete graph,
    $$\chi_{(n,k)}(K_m) =
    \begin{cases}
        m, &\text{if } (m-1,n)\mid k, \\
        \text{does not exist}, &\text{otherwise.}
    \end{cases}
    $$
\end{theorem}
\begin{proof}
    Suppose $\ell$ is a proper open coloring with remainder $k\mod{n}$ of $K_m$. Then for every $v,w \in V$, Lemma \ref{lem: venn} requires $\ell(v) \equiv \ell(w)\mod n$.
    In turn, the open neighborhood condition requires $(m-1)\ell(v) \equiv k \mod n$. Thus, $(m-1,n)\mid k$.

    Conversely, if $(m-1,n)\mid k$, then there is some $\alpha\in\Z$ with $(m-1)\alpha \equiv k \mod n$. A labeling of the vertices by $\alpha,\alpha+n,\dots,\alpha+(m-1)n$ gives a proper open coloring of order $\chi(K_m)=m$ with remainder $k\mod{n}$.
\end{proof}

\begin{theorem}
    Let $k\in \Z$  and $n,m\in\Z^+$ with $n\geq 2$, $k\not\equiv 0\mod n$, and $m \geq 2$. For the star,
    $$\chi_{(n,k)}(S_m) =
    \begin{cases}
        2, &\text{if } (m,n)\mid k, \\
        3, &\text{otherwise}.
    \end{cases}
    $$
\end{theorem}

\begin{proof}
    By definition, any open coloring with remainder $k\mod n$ of $S_m$ must label the central vertex with a label congruent to $k \mod n$.
    Observe that $\chi_{(n,k)}(S_m)=2$ if and only if there exists some $\alpha\in \Z$ such that $m\alpha \equiv k \mod n$. This is equivalent to $(m,n)\mid k$.

    Otherwise, a proper open $3$-coloring with remainder $k\mod n$ of $S_m$ may be obtained as follows: Color the central vertex with $k$. Then color exactly one circumferential vertex with $k+n$ and the rest with~$0$.
\end{proof}

\begin{theorem}
Let $k\in \Z$ and $n, i,j \in \Z^+$ with $n\geq 2$ and $k\not\equiv 0\mod n$. For the complete bipartite graph,
$$
\chi_{(n,k)}(K_{i,j}) =
\begin{cases}
2, &\text{if } (i,n) \mid k \text{ and } (j,n) \mid k, \\
3, &\text{if } (i,n) \mid k \text{ or } (j,n) \mid k, \text{ but not both}, \\
4, &\text{otherwise.}
\end{cases}
$$
\end{theorem}
\begin{proof}
    Let $V_1$ and $V_2$ with $|V_1|=i$ and $|V_2|=j$ denote the vertex sets belonging to the two parts of $K_{i,j}$.
    Labeling exactly one vertex of $V_1$ and exactly one vertex of $V_2$ with $k$ and the rest with $0$ gives an open coloring with remainder $k\mod n$. Theorem \ref{thm: remainder k mod n existence criteria and upper bound by n chi} shows that $2\leq \chi_{(n,k)}(K_{i,j}) \leq 4$.

    We have $\chi_{(n,k)}(K_{i,j})=2$ if and only if all vertices of $V_1$ can be labeled with the same label $\alpha$ and all vertices of $V_2$ with the same label $\beta$ for distinct $\alpha,\beta \in \Z$ with $i\alpha \equiv k \mod n$ and $j\beta \equiv k \mod n$. This is possible if and only if $(i,n), (j,n) \mid k$.

    If $(i,n) \mid k$, but $(j,n) \nmid k$, choose some $k\ne \alpha \in \Z$ such that $i\alpha\equiv k \mod n$. Then a proper open $3$-coloring with remainder $k\mod n$ of $K_{i,j}$ is obtainable by labeling one vertex of $V_2$ with $k$ and the rest with $0$ and all the vertices of $V_1$ with $\alpha$. The case $(i,n) \nmid k$, but $(j,n) \mid k$, is done similarly.

    If $(i,n), (j,n) \nmid k$, then labeling $V_1$ needs at least two colors as does labeling $V_2$. But as these colors must be mutually distinct to get a proper coloring, we are done.
\end{proof}

\begin{lemma}
Let $k \in \Z$ and $n \in\Z^+$ with $n\geq 2$ and $k\not\equiv 0\mod n$. Let $R_4$ denote the regular, infinite square tiling of the plane. Then
$2 \leq \chi_{(n,k)}(R_4) \leq 4$,
and $\chi_{(n,k)}(R_4) = 2$ if and only if $(4,n)\mid k$.
\label{Square Tiling Bounds - Open Neighborhoods}
\end{lemma}
\begin{proof}
    Write $V=\{v_{i,j} \mid (i,j)\in\Z\times\Z\}$ for the vertices of $R_4$.
    Consider the labeling defined as follows.
    \begin{enumerate}
        \item If $2 \mid i$,
        $$\ell(v_{i,j}) =
        \begin{cases}
            0, &\text{if } 2\mid j, \\
            n, &\text{otherwise}.
        \end{cases}
        $$
        \item If $i \equiv 1 \mod 4$,
        $$\ell(v_{i,j}) =
        \begin{cases}
            n, &\text{if } j \equiv 0\mod{4}, \\
            0, &\text{if } j \equiv 1\mod{4}, \\
            k, &\text{if } j \equiv 2\mod{4}, \\
            k+n, &\text{otherwise}.
        \end{cases}
        $$
        \item If $i \equiv 3 \mod 4$,
        $$\ell(v_{i,j}) =
        \begin{cases}
            k, &\text{if } j \equiv 0\mod{4}, \\
            k+n, &\text{if } j \equiv 1\mod{4}, \\
            n, &\text{if } j \equiv 2\mod{4}, \\
            0, &\text{otherwise}.
        \end{cases}
        $$
    \end{enumerate}
    Since this is a proper open coloring of order $4$ with remainder $k \mod{n}$, see Figure \ref{Coloring of S - Order 4}, we get $\chi_{(n,k)}(R_4) \leq 4$.

\begin{figure}[H]
\begin{center}
\begin{tikzpicture}
\def\numrows{7} 
\def\numcols{7} 
\def\squaresize{1.5}

    \foreach \row in {-3,...,3}{
        \foreach \col in {-3,...,3}{
            \pgfmathtruncatemacro{\y}{mod(mod(\row, 4) + 4, 4)}
            \pgfmathtruncatemacro{\yeven}{mod(mod(\y, 2) + 2, 2)}
            \pgfmathtruncatemacro{\x}{mod(mod(\col, 4) + 4, 4)}
            \pgfmathtruncatemacro{\xeven}{mod(mod(\x, 2) + 2, 2)}
            \filldraw (\col*\squaresize,\row*\squaresize) circle (2pt);
            \ifnum\x=0
                \ifnum\y=0
                    \node[anchor=south west, inner sep=2pt] at (\col*\squaresize,\row*\squaresize) {$0$};
                    \node[anchor=north east, inner sep=2pt] at (-1.5,-1.5) {$(0,0)$};
                \fi
                \ifnum\y=1
                    \node[anchor=south west, inner sep=2pt] at (\col*\squaresize,\row*\squaresize) {$k$};
                \fi
                \ifnum\y=2
                    \node[anchor=south west, inner sep=2pt] at (\col*\squaresize,\row*\squaresize) {$k+n$};
                \fi
                \ifnum\y=3
                    \node[anchor=south west, inner sep=2pt] at (\col*\squaresize,\row*\squaresize) {$n$};
                \fi
        \fi
        \ifnum\x=2
                \ifnum\y=0
                    \node[anchor=south west, inner sep=2pt] at (\col*\squaresize,\row*\squaresize) {$k+n$};
                \fi
                \ifnum\y=1
                    \node[anchor=south west, inner sep=2pt] at (\col*\squaresize,\row*\squaresize) {$n$};
                \fi
                \ifnum\y=2
                    \node[anchor=south west, inner sep=2pt] at (\col*\squaresize,\row*\squaresize) {$0$};
                \fi
                \ifnum\y=3
                    \node[anchor=south west, inner sep=2pt] at (\col*\squaresize,\row*\squaresize) {$k$};
                \fi
        \fi
        \ifnum\xeven=1
                \ifnum\yeven=0
                    \node[anchor=south west, inner sep=2pt] at (\col*\squaresize,\row*\squaresize) {$n$};
                \fi
                \ifnum\yeven=1
                    \node[anchor=south west, inner sep=2pt] at (\col*\squaresize,\row*\squaresize) {$0$};
                \fi
        \fi
        }
        \node[anchor=west, inner sep=10pt] at (3*\squaresize,\row*\squaresize) {$\cdots$};
        \node[anchor=east, inner sep=10pt] at (-3*\squaresize,\row*\squaresize) {$\cdots$};
    }
    \foreach \col in {-3,...,3}{
        \node[anchor=south, inner sep=10pt] at (\col*\squaresize,3*\squaresize) {$\vdots$};
        \node[anchor=north, inner sep=10pt] at (\col*\squaresize,-3*\squaresize) {$\vdots$};
    }

    \foreach \row in {-3,...,3}{
        \foreach \col in {-3,...,3}{
            \ifnum\col>-3
                \draw (\col*\squaresize,\row*\squaresize) -- ({(\col-1)*\squaresize},{\row*\squaresize});
            \fi
            \ifnum\row>-3
                \draw (\col*\squaresize,\row*\squaresize) -- ({\col*\squaresize},{(\row-1)*\squaresize});
            \fi
        }
    }
\end{tikzpicture}
\end{center}
\caption{A Proper Open Coloring of $R_4$ of Order $4$}
\label{Coloring of S - Order 4}
\end{figure}
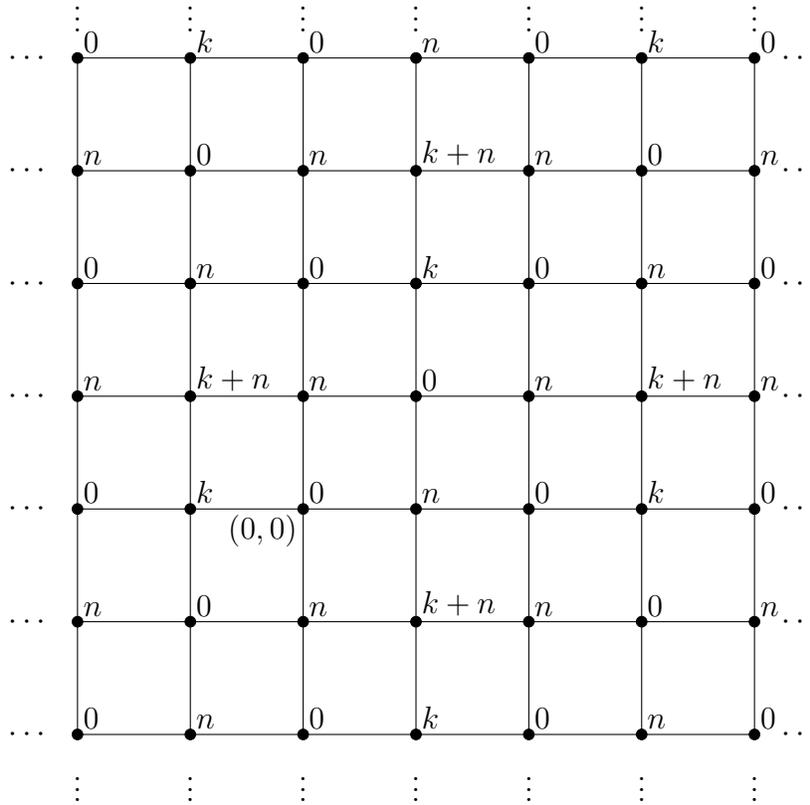

    Finally, $\chi_{(n,k)}(R_4) = 2$ if and only if $v_{i,j}$ is labeled according to the parity of $i+j$ with $\alpha, \beta \in \Z$, respectively, such that $\alpha \ne \beta$ and $4\alpha \equiv 4\beta \equiv k \mod n$. This is possible if and only if $(4,n)\mid k$.
\end{proof}

\begin{theorem} \label{Square Tiling - Open Neighborhoods}
Let $k \in \Z$ and $n \in\Z^+$ with $n\geq 2$ and $k\not\equiv 0\mod n$. Let $R_4$ be the regular, infinite square tiling of the plane. Then
$$
\chi_{(n,k)}(R_4) =
\begin{cases}
2, &\text{if } (4,n) \mid k, \\
4, &\text{otherwise.}
\end{cases}
$$
\end{theorem}
\begin{proof}
    We continue our notation from Lemma \ref{Square Tiling Bounds - Open Neighborhoods}.
    It remains to show that the existence of a proper open $3$-coloring with remainder $k\mod n$ of $R_4$ forces $(4,n)\mid k$. To that end, suppose that $\ell$ is such a coloring with distinct labels $\alpha,\beta,\gamma \in \Z$.

    Let $R_4 =(V,E)$. If $|\ell(N(v))|\ge 3$ for any $v\in V$, then there is no possible label left for the vertex $v$. Thus, for all $v\in V$, $|\ell(N(v))|\leq 2$. If $|\ell(N(v))|=1$ for some $v$, the open neighborhood sum condition for $v$ implies that we can solve the equation $4x\equiv k \mod n$. As this requires $(4,n)\mid k$, we may reduce to the case where $|\ell(N(v))|=2$ for all $v\in V$.

    Consider the case where there exists some $v\in V$ such that a color appears three times in $N(v)$. After relabeling, w.l.o.g. we may assume that $\ell(v_{0,0})=\alpha$, $\ell(v_{-1,0}) = \ell(v_{0,-1}) = \ell(v_{1,0}) = \beta$, and $\ell(v_{0,1}) = \gamma$. Then $\ell(v_{-1,1}) = \ell(v_{1,1}) = \alpha$ and so, as $|\ell(N(v_{0,1}))|=2$, $\ell(v_{0,2}) = \beta$. Similarly, it follows that $\ell(v_{-1,2}) = \ell(v_{1,2}) = \gamma$ and $\ell(v_{0,3}) = \alpha$. Summing the open neighborhood sums at $v_{0,0}, v_{0,1}, v_{0,2}$ now shows that $4(\alpha+\beta+\gamma) \equiv 3k\mod n$. In turn, this requires $(4,n)\mid(3k)$ so that $(4,n)\mid k$.

    As a result, we are reduced to the case where, for all $v\in V$, each color that appears in $N(v)$ appears exactly twice. However, summing the open neighborhood sums at a vertex labeled by $\alpha$, one by $\beta$, and one by $\gamma$ gives again $4(\alpha+\beta+\gamma) \equiv 3k \mod n$ so that $(4,n)\mid k$.
\end{proof}

We will write $T^*_m$ for the \emph{(infinite) regular tree of degree $m$} so that the degree of each vertex is $m$. We will fix a vertex of $T^*_m$, $v_0$, and view it as the root. In that case, for any vertex $v$ of $T^*_m$, write $h(v)$ for the height of $v$, i.e., the distance from $v$ to the root $v_0$.

\begin{theorem}\label{thm: infinite regular trees}
    Let $k\in\Z$ and $n,m\in\Z^+$ with $n\geq 2$, $k\not\equiv 0\mod n$, and $m\geq 2$. For the regular tree of degree $m$,
    \[
        \chi_{(n,k)}(T^*_m) =
        \begin{cases}
            2, & \text{ if } (m,n)\mid k, \\
            3 \text{ or } 4, & \text{ otherwise.}
        \end{cases}
    \]
\end{theorem}

\begin{proof}
    Beginning with $v_0$ labeled by $0$ and inducting on the height, it is always possible to label $T^*_m$ with $\{0,k\}$ to get an open coloring with remainder $k\mod n$. By Theorem \ref{thm: remainder k mod n existence criteria and upper bound by n chi}, we have $2 \leq \chi_{(n,k)}(T^*_m) \leq 4$.

    Now a proper open $2$-coloring with remainder $k\mod n$ of $T^*_m$ exists if and only if $T^*_m$ can be labeled according to the parity of $h(v)$ with $\alpha, \beta \in \Z$, respectively, such that $\alpha \ne \beta$ and $m\alpha\equiv m\beta\equiv k \mod n$. This is possible if and only if $(m,n)\mid k$.
\end{proof}

\begin{question}
In Theorem \ref{thm: infinite regular trees}, it is not known if $\chi_{(n,k)}(T^*_m) = 3$ is possible. In the simplest case, $m=2$, it is straightforward to see that when $(2,n)\nmid k$, then $\chi_{(n,k)}(T^*_2)=4$ with a repeating pattern of labels $(\ldots,\alpha, \beta, \gamma, \delta,\ldots)$ with $\gamma \equiv (k-\alpha) \mod n$ and $\delta \equiv (k-\beta) \mod n$.
\end{question}

 We write $T_m$ for the \emph{(rooted) complete $m$-ary tree of infinite height}. We continue to write $h(v)$ for the distance from the vertex $v$ of $T_m$ to its root, $v_0$.

\begin{theorem} \label{infinite trees - open neighborhoods}
Let $k\in \Z$ and $n,m\in\Z^+$ with $n\geq 2$, $k\not\equiv 0\mod n$, and $m\geq 1$. For the complete $m$-ary tree of infinite height,
\[
        \chi_{(n,k)}(T_m)=
            \begin{cases}
                3, & \text{ if } (m+1,n)\mid k, \\
                3 \text{ or } 4, & \text{otherwise.}
            \end{cases}
\]
\end{theorem}

\begin{proof}
    Beginning with $v_0$ labeled by $0$ and inducting on the height, it is always possible to label $T_m$ with $\{0,k\}$ to get an open coloring with remainder $k\mod n$. As a result, $2\leq \chi_{(n,k)}(T_m) \leq 4$ by Theorem \ref{thm: remainder k mod n existence criteria and upper bound by n chi}.

    A proper open $2$-coloring with remainder $k\mod n$ of $T_m$ exists if and only if $T_m$ can be labeled according to the parity of $h(v)$ with $\alpha, \beta \in \Z$, respectively, such that $\alpha \ne \beta$, $m\beta \equiv k\mod n$, $(m+1)\alpha \equiv k \mod n$, and $(m+1)\beta\equiv k \mod n$. As $m\beta \equiv (m+1)\beta \mod n$ implies $\beta \equiv 0\mod n$ and $k\equiv 0 \mod n$, it is not possible.

    If $(m+1,n)\mid k$, choose some $k\ne \alpha\in \Z$ such that $(m+1)\alpha \equiv k\mod n$. Then a proper open $3$-coloring with remainder $k\mod n$ of $T_m$ is achievable by labeling all vertices of even height with $\alpha$. For vertices of height $1$, label one vertex with $k$ and the remainder with $0$. Induct on the height by labeling all grandchildren of a vertex labeled by $k$ with $0$. For grandchildren of a vertex labeled by $0$, label one with $k$ and the remainder with $0$.
\end{proof}

\begin{remark}
In Theorem \ref{infinite trees - open neighborhoods}, $\chi_{(n,k)}(T_m)=3$ can be achieved also for some cases where $(m+1,n)\nmid k$.
Indeed, for $k=1$ and $n=3$ with $3\mid (m+1)$, a proper open $3$-coloring with remainder $1 \mod 3$ of $T_m$ is achieved by a constant row labeling according to the repeated pattern $1,-1,0$ starting from the root $v_0$.
\end{remark}

We write $T_{m,d}$ for the \emph{(rooted) complete $m$-ary tree of height $d$}. We continue to write $h(v)$ for the distance from the vertex $v$ of $T_m$ to its root, $v_0$. We also write $r(v):=d-h(v)$ for the \emph{reverse height}.
\begin{theorem}
    Let $k\in \Z$ and $n,m,d\in\Z^+$ with $n\geq 2$ and $k\not\equiv 0\mod n$.
    Write $\delta = \lfloor \frac{d}{2} \rfloor$.

    If $d$ is even, then $\chi_{(n,k)}(T_{m,d})$ exists if and only if
    \[ n\mid \left( k \, \frac{m^{\delta +1}+ (-1)^{\delta}}{m+1} \right). \]
    In that case,
    \[ \chi_{(n,k)}(T_{m,d}) \leq d+1.\]

    If $d$ is odd, then $\chi_{(n,k)}(T_{m,d})$ always exists.
    If
    \[ n\mid \left( k \, \frac{m^{\delta +1}+ (-1)^{\delta}}{m+1} \right), \]
    then
    \[ \chi_{(n,k)}(T_{m,d}) \leq d + 1.\]
    Otherwise,
    \[ \chi_{(n,k)}(T_{m,d}) \leq d+\delta+2.\]
\end{theorem}

\begin{proof}
    Recall that $r$ denotes the reverse height function. By definition, in any open coloring with remainder $k\mod n$ of $T_{m,d}$, the labels of vertices $v$ with $r(v)=0$ inductively determine the labels of all vertices $v$ with $r(v)\in 2\Z$, and the labels of vertices $v$ with $r(v)=1$ must be congruent to $k \mod n$ and inductively determine the labels of all vertices $v$ with $r(v)\in 2\Z+1$. After that, there will only be one open neighborhood sum to be checked, at the root $v_0$.

    As the label of any vertex $v$ with $r(v)=1$ is congruent to $k\mod n$, it follows that each odd reverse height row of $T_{m,d}$ consists of congruent labels $\mod n$. Let $x_i$ denote the label of some vertex $v$ with $r(v)=2i+1$, $1\leq 2i+1 \leq d$. Then $x_0 \equiv k \mod n$ and $(x_{i}+m x_{i-1}) \equiv k \mod n$ for all $3\le 2i+1 \le d$, a linear recurrence relation.

    Our inhomogeneous linear recurrence relation $(x_{i}+m x_{i-1}) \equiv k \mod n$ leads to the second-order homogeneous linear recurrence relation \[x_{i}+(m-1) x_{i-1}-mx_{i-2} \equiv 0 \mod n\] with the initial conditions of $x_0 \equiv k \mod n$ and $x_1 \equiv k(1-m) \mod n$. Solving this recurrence relation, we see that
    \[ x_i \equiv k\, \frac{1 -(-m)^{i+1}}{m+1} \mod n. \]

    In particular, when $d=2\delta$ is even, the final constraint of having a $k\mod n$ open neighborhood sum at $v_0$ becomes
    \[ k\equiv m x_{\delta -1} \equiv  km\, \frac{1 -(-m)^{\delta}}{m+1} \mod n.\]
    Rewriting gives that an open coloring with remainder $k\mod n$ of $T_{m,d}$ cannot exist if
    \[ k\, \frac{1 +m(-m)^{\delta}}{m+1} \equiv (-1)^\delta k\, \frac{m^{\delta+1}+(-1)^\delta}{m+1} \not\equiv 0\mod n.\]
    In all other cases of $d$ and $m$, an open coloring with remainder $k\mod n$ for rows of odd reverse height can be achieved through a constant row labeling,
    which requires at most $\delta$ distinct labels for even $d$ and at most $\delta +1$ distinct labels for odd $d$.

    It remains to discuss an open coloring with remainder $k\mod n$ for rows of even reverse height. We will see that such a coloring can always be obtained by labeling all vertices $v$ with $r(v)=0$, up to congruence $\mod n$, with $0$,
    except possibly for one vertex $v^*$ labeled with $\alpha_0$. In the following, we will discuss how this initial condition affects the labels of all other vertices $v$ with $r(v) \in 2\Z$.


    Again, it is easy to see that all the vertices of any even reverse height row of $T_{m,d}$ that lie not on the shortest path from $v^*$ to $v_0$ must share congruent labels $\mod n$. Let $y_i$ denote the label of some vertex $v$ with $r(v)=2i$, $0\leq 2i < d$, that does not lie on the shortest path from $v^*$ to~$v_0$. Then $y_0 \equiv 0 \mod n$ and $(y_{i}+m y_{i-1}) \equiv k \mod n$ for all $2\le 2i < d$. We see that $y_1 \equiv k \mod n$, hence
    \[ y_i \equiv x_{i-1} \equiv k\, \frac{1 -(-m)^{i}}{m+1} \mod n.  \]
    Similarly, for the vertex $v$ with $r(v)=2i$ on the shortest path from $v^*$ to $v_0$, one finds a label congruent to $(y_i + (-1)^i \alpha_0) \mod n$.

    In particular, when $d=2\delta +1$ is odd, the final constraint of having a $k\mod n$ open neighborhood sum at $v_0$ becomes
    \[ k \equiv m y_{\delta} + (-1)^\delta \alpha_0 \equiv \left( km\, \frac{1 -(-m)^{\delta}}{m+1} + (-1)^\delta \alpha_0\right) \mod n,\]
    which can always be solved for $\alpha_0$ and gives
    \[\alpha_0 \equiv k\, \frac{m^{\delta+1}+(-1)^\delta}{m+1} \mod n.\]
    Thus, an open coloring with remainder $k\mod n$ for rows of even reverse height is always possible.
    If $d=2\delta +1$ is odd with
    \[k\, \frac{m^{\delta+1}+(-1)^\delta}{m+1} \not\equiv 0 \mod n,\]
    we must choose $\alpha_0 \not\equiv 0 \mod n$ and two distinct labels per row may be necessary for the open coloring. Thus, in this case, we succeed with a proper open coloring with remainder $k\mod n$ of $T_{m,d}$ by using at most $2(\delta+1)$ distinct labels for the rows of even reverse height in addition to the labels used for rows of odd reverse height.
    In all other cases of $d$ and $m$, we can choose $\alpha_0 \equiv 0 \mod n$, and an open coloring with remainder $k\mod n$ for rows of even reverse height can be achieved through a constant row labeling.
    This will require at most an additional $\delta + 1$ distinct labels for a proper coloring.
\end{proof}

We turn now to the generalized Petersen graph $G(m,j)$. We will use the same notation as for Theorem \ref{thm: chi n gen petersen}.

\begin{theorem} \label{thm: gen pet n k}
    Let $k \in \Z$ and $n,m,j\in\Z^+$ with $n\geq 2$, $k\not\equiv 0 \mod n$, $m\geq 3$, and $1\leq j < \frac{m}{2}$. Then the following holds:
    \begin{itemize}
    \item If $(3,n)\mid k$, then
        $\chi_{(n,k)}(G(m,j)) = \chi(G(m,j))$.

    \item If $(3,n) \nmid (km)$, then $\chi_{(n,k)}(G(m,j))$ does not exist.

    \item If $(3,n)\nmid k$, $(3,n) \mid (km)$, and $3 \nmid j$, then $\chi_{(n,k)}(G(m,j))$ exists and
    \[  \chi(G(m,j)) \leq \chi_{(n,k)}(G(m,j)) \leq 2 \chi(G(m,j)). \]
    \end{itemize}
\end{theorem}

\begin{proof}
    The cases of $(3,n)\mid k$ and $(3,n) \nmid (km) \iff (3,n) \nmid (2km)$ are handled by Theorem \ref{thm: regular chi k n}. Therefore assume $(3,n)\nmid k$ and $(3,n) \mid (km)$, which is equivalent to $3\mid n$, $3\nmid k$, and $3\mid m$.

    If $3\nmid j$, labeling $v_i, u_i$ with $k$ for $3\mid i$ and with $0$ otherwise gives an open $2$-coloring with remainder $k\mod n$ of $G(m,j)$. By Theorem \ref{thm: remainder k mod n existence criteria and upper bound by n chi},
    \[ \chi(G(m,j)) \leq \chi_{(n,k)}(G(m,j)) \leq 2\chi(G(m,j)).\qedhere \]
\end{proof}

\begin{remark}
Theorem \ref{thm: gen pet n k} still leaves open the case $3\mid n$, $3\nmid k$, $3\mid m$, and $3\mid j$. In this case, a more detailed analysis shows that a proper open coloring with remainder $k\mod n$  of $G(m,j)$ also exists if either $9\mid m$ or $9\nmid m$ and $9\nmid n$. However, general results remain unknown.
\end{remark}

\section{Basic Results for $\chi_{(n)}(G)$}

As we will see in Section \ref{sec:8}, $\chi_{(n)}(G)$ may not exist. However, there are bounds when it does. See Section \ref{sec:3} for corresponding results on~$\chi_{n}(G)$.

\begin{theorem}\label{thm: open nonzero mod n inequalities}
    Let $m,n\in\Z^+$ with $m\mid n$. If $\chi_{(m)}(G)$ exists, then
    \[ \chi(G) \leq \chi_{(n)}(G) \leq \chi_{(m)}(G). \]
\end{theorem}

\begin{proof}
    The first inequality follows immediately from the definition. For the second, observe that if $\ell$ is an open coloring with nonzero remainders $\mod m$, then it is also one for $\mod n$.
\end{proof}

\begin{theorem}\label{thm: open nonzero remainder mod n existence criteria and upper bound by n chi}
    Let $n\in\Z^+$, and let $\chi(G)$ be finite. Then a proper open coloring with nonzero remainders $\mod n$ of $G$ exists if and only if an open coloring with nonzero remainders $\mod n$ of $G$ exists. In that case,
    \[ \chi(G)\le \chi_{(n)}(G) \leq n \, \chi(G). \]
    More precisely, if $\ell$ is an open labeling with nonzero remainders $\mod n$, then
    \[  \chi(G)\le \chi_{(n)}(G) \leq |\ell| \, \chi(G).\]
\end{theorem}

\begin{proof}
    Let $\ell$ be an open coloring with nonzero remainders $\mod n$ of $G$ and let $\ell'$ be a minimal proper labeling of $G$. We may assume that the range of $\ell$ sits in $[0,n-1]$, and we may assume that the range of $\ell'$ sits in $n\Z$. Then the labeling $\ell + \ell'$ is a proper open coloring with nonzero remainders $\mod n$ of $G$. As its order is bounded by $|\ell| \chi(G)$ and since $|\ell|\leq n$, we are done.
\end{proof}

\begin{theorem}\label{thm: open regular chi n}
    Let $n,j\in\Z^+$, and let $G$ be a $j$-regular graph. Then
    \[ n \nmid j \implies \chi_{(n)}(G) = \chi(G). \]
\end{theorem}

\begin{proof}
    If $n \nmid j$, then a constant labeling of $G$ by $1$ is an open coloring with nonzero remainders $\mod n$. Theorem \ref{thm: open nonzero remainder mod n existence criteria and upper bound by n chi} finishes the proof.
\end{proof}

\section{Examples for $\chi_{(n)}(G)$} \label{sec:8}

In the following, we evaluate $\chi_{(n)}(G)$ for a few special graph families. See Section \ref{sec:4} for the corresponding results on~$\chi_{n}(G)$.

\begin{theorem}
    Let $n,m\in\Z^+$ with $n,m\geq 2$. For the complete graph,
    $$\chi_{(n)}(K_m) =
    \begin{cases}
        m, &\text{if } n > 2 \text{ or } n=2 \text{ with } 2\mid m, \\
        \text{does not exist}, &\text{otherwise}.
    \end{cases}
    $$
\end{theorem}

\begin{proof}
    For $n>2$, labeling the vertices of $K_m$ with $1$, $n+1$, and $n, 2n, \dots, (m-2)n$ gives an open coloring with nonzero remainders $\mod n$. For $n=2$, the open colorings with nonzero remainders $\mod 2$ are identical to the open colorings with remainder $1\mod 2$. Thus, we have $\chi_{(2)}(K_m) = \chi_{(2,1)}(K_m)$, and Theorem \ref{Complete Graphs - Open Neighborhoods} applies.
    %
    %
    %
\end{proof}

\begin{theorem}
Let $n,i,j \in \Z^+$ with $n\geq 2$. For the complete bipartite graph,
\[
\chi_{(n)}(K_{i,j}) =
\begin{cases}
2, &\text{if } n \nmid i \text{ and } n \nmid j, \\
3, &\text{if } n \nmid i \text{ or } n \nmid j, \text{ but not both}, \\
4, &\text{otherwise}.
\end{cases}
\]
\end{theorem}
\begin{proof}
    Let $V_1$ and $V_2$ with $|V_1|=i$ and $|V_2|=j$ denote the vertex sets belonging to the two parts of $K_{i,j}$. By labeling a single vertex from $V_1$ and a single vertex from $V_2$ with $1$ and the rest with $0$, Theorem \ref{thm: open nonzero remainder mod n existence criteria and upper bound by n chi} shows that $2 \leq \chi_{(n)}(K_{i,j}) \leq 4$.

    We have $\chi_{(n)}(K_{i,j})=2$ if and only if all vertices of $V_1$ can be labeled with the same label $\alpha$ and all vertices of $V_2$ with the same label $\beta$ for distinct $\alpha,\beta \in \Z$ with $i\alpha \not\equiv 0 \mod n$ and $j\beta \not \equiv 0 \mod n$. This is possible if and only if $n\nmid i$ and $n\nmid j$.

    If $n\mid i$ and $n \nmid j$, then a proper open $3$-coloring with nonzero remainders $\mod n$ of $K_{i,j}$ is obtainable by labeling one vertex of $V_1$ with $1$ and the rest with $0$ and all the vertices of $V_2$ with $n+1$. The case where $n\nmid i$ and $n \mid j$ is done similarly.

    If $n\mid i$ and $n\mid j$, then labeling $V_1$ needs at least two colors as does labeling $V_2$. But as these colors must be mutually distinct to get a proper coloring, we are done.
\end{proof}

\begin{theorem}
Let $n,m \in \Z^+$ with $n \geq 2$ and $m \geq 1$. For the complete $m$-ary tree of infinite height,
$$
\chi_{(n)}(T_m) =
\begin{cases}
2, &\text{if } n \nmid m \text{ and } n \nmid (m+1), \\
4, &\text{if } n = 2 \text{ and $m$\! odd}, \\
3, &\text{otherwise}.
\end{cases}
$$
\end{theorem}

\begin{proof}
    Recall that we write $v_0$ for the root of $T_m$ and $h(v)$ for the height of the vertex $v$ of $T_m$. An open coloring with nonzero remainders $\mod n$ of $T_m$ can be obtained by the following algorithm. Label $v_0$ with $0$, one of its children with $1$, and the rest of the children with $0$. Now induct on the height by labeling all grandchildren of any vertex labeled by $1$ with $0$. For grandchildren of a vertex labeled by $0$, label one with $1$ and the remainder with $0$. As a result, Theorem \ref{thm: open nonzero remainder mod n existence criteria and upper bound by n chi} shows that $2\leq \chi_{(n)}(T_m) \leq 4$.

    Now $\chi_{(n)}(T_m)=2$ if and only if $T_m$ can be labeled according to the parity of $h(v)$ with $\alpha, \beta \in \Z$, respectively, such that
    $m\beta\not\equiv 0 \mod n$, $(m+1)\alpha\not\equiv 0 \mod n$, and $(m+1)\beta\not \equiv 0 \mod n$. This is possible if and only if $n \nmid m$ and $n \nmid (m+1)$.

    If $n \mid m$, a proper open $3$-coloring with nonzero remainders $\mod n$ of $T_m$ is given by the following algorithm. Label all vertices of $T_m$ of even height with $1$, label one child of $v_0$ with $n+1$, and the rest of the children with $0$. We continue to inductively label the remaining vertices  of odd height of $T_m$ by labeling all grandchildren of any vertex labeled by $n+1$ with $0$. For grandchildren of a vertex labeled by $0$, label one with $n+1$ and the remainder with $0$.

    This leaves the case of $n \mid (m+1)$. If $n\geq 3$, a proper open $3$-coloring with nonzero remainders $\mod n$ of $T_m$ is achieved by a constant row labeling according to the repeated pattern $1,-1,0$ starting from the root $v_0$.

    However, if $n \mid (m+1)$ and $n=2$, we claim that there is no such proper open $3$-coloring. If there were, suppose that $\alpha_i\in \Z$ for $i \in \{1,2,3\}$ are the three distinct labels used, where indices will be treated as members of $\Z_3$ for convenience. As $m+1\equiv 0 \mod n$, any non-root vertex of $T_m$ must have at least one child with a different label than their parent. As a result, we can find non-root vertices $v_i$ labeled with $\alpha_i$ for each $i\in \{1,2,3\}$. The open neighborhood sum condition for $v_i$ then requires $c_i \alpha_{i+1} + d_i\alpha_{i+2}\equiv 1 \mod 2$ for integers $c_i,d_i\in \Z^+$ with $c_i+d_i=m+1$. This implies $c_i+d_i \equiv m+1\equiv 0 \mod 2$. Hence $c_i\equiv d_i \mod 2$, and $c_i \alpha_{i+1} + d_i\alpha_{i+2} \equiv c_i(\alpha_{i+1} + \alpha_{i+2})\equiv 1 \mod 2$ follows. In particular, $\alpha_{i+1} + \alpha_{i+2}\equiv 1 \mod 2$. Summing over $i \in \{1,2,3\}$, we get
    $2\sum_{i=1}^3 \alpha_i \equiv 1 \mod 2$, a contradiction.
\end{proof}

\section{Concluding Remarks}


It is worth pointing out that this paper only scratches the surface for the three introduced graph invariants, $\chi_n(G)$, $\chi_{(n,k)}(G)$, and $\chi_{(n)}(G)$. There is much room for further exploration and more precise results. Some avenues of investigation of particular interest include:

\begin{itemize}
    \item Finding conditions under which $\chi_n(G)$, $\chi_{(n,k)}(G)$, and $\chi_{(n)}(G)$ are equal to (or close to) $n\chi(G)$,
    \item Finding exact values for $\chi_n(G(m,j))$ and $\chi_{(n,k)}(G(m,j))$ for the values of $k,n,m,j$ not covered in our Theorems \ref{thm: chi n gen petersen} and \ref{thm: gen pet n k}, and
    \item Finding exact values for $\chi_n(G)$, $\chi_{(n,k)}(G)$, and $\chi_{(n)}(G)$ for more families of graphs.

\end{itemize}

\bibliographystyle{abbrvnat}
\bibliography{refs}

\end{document}